\documentclass{article}

\usepackage{amsmath, amssymb, amsthm, hyperref, setspace, dsfont,
  kpfonts, url}
\usepackage[T1]{fontenc}

\hypersetup{
    colorlinks,
    citecolor=black,
    filecolor=black,
    linkcolor=black,
    urlcolor=black
}

\usepackage{fancyhdr} 
 
\newtheorem{theorem}{Theorem}[section]
\newtheorem{corollary}[theorem]{Corollary}
\newtheorem{lemma}[theorem]{Lemma}
\newtheorem{proposition}[theorem]{Proposition}

 \theoremstyle{definition}
 \theoremstyle{remark}
\newtheorem{remark}[theorem]{Remark} 

\numberwithin{equation}{section} 

\newcommand{\comment}[1]{}

\newcommand{\R}{\mathbb R}

\newcommand{\eps}{\varepsilon}

\newcommand{\ls}{\leqslant}
\newcommand{\gr}{\geqslant}

\providecommand{\abs}[1]{\lvert#1\rvert}
\providecommand{\norm}[1]{\left\lVert#1\right\rVert}

\providecommand{\wtd}[1]{\widetilde#1}

\newcommand\blfootnote[1]{%
  \begingroup
  \renewcommand\thefootnote{}\footnote{#1}%
  \addtocounter{footnote}{-1}%
  \endgroup
}

\begin{document}

\title{On a quantitative reversal of Alexandrov's inequality}

\author{Grigoris Paouris\thanks{Supported by NSF grant
    CAREER-1151711.} \and Peter Pivovarov\thanks{Supported by NSF
    grant DMS-1612936.} \and Petros Valettas\footnotemark[2]}

\maketitle

\begin{abstract}
  Alexandrov's inequalities imply that for any convex body $A$, the
  sequence of intrinsic volumes $V_1(A),\ldots,V_n(A)$ is
  non-increasing (when suitably normalized).  Milman's random version
  of Dvoretzky's theorem shows that a large initial segment of this
  sequence is essentially constant, up to a critical parameter called
  the Dvoretzky number.  We show that this near-constant behavior
  actually extends further, up to a different parameter associated with
  $A$. This yields a new quantitative reverse inequality that sits
  between the approximate reverse Urysohn inequality, due to
  Figiel--Tomczak-Jaegermann and Pisier, and the sharp reverse Urysohn
  inequality for zonoids, due to Hug--Schneider.  In fact, we study
  concentration properties of the volume radius and mean width of
  random projections of $A$ and show how these lead naturally to such
  reversals.
\end{abstract}


\section{Introduction}
For a convex body $A\subseteq \R^n$, the intrinsic volumes $V_1(A),
\ldots, V_n(A)$ are fundamental quantities in convex geometry. Of
special significance are $V_1$, $V_{n-1}$ and $V_n$, which are
suitable multiples of the mean width, surface area and volume,
respectively (precise definitions are recalled in \S 2). Alexandrov's
inequalities imply that
\begin{equation}
  \label{eqn:Alexandrov}
  \left(\frac{V_n(A)}{V_n(B)}\right)^{\frac{1}{n}}\ls
  \left(\frac{V_{n-1}(A)}{V_{n-1}(B)}\right)^{\frac{1}{n-1}}\ls\ldots
  \ls \frac{V_1(A)}{V_1(B)},
\end{equation}
where $B$ is the Euclidean unit ball in $\R^n$. The leftmost
inequality is the isoperimetric inequality, while Urysohn's inequality
is the comparison between the two endpoints. Thus
\eqref{eqn:Alexandrov} occupies a special role in convex geometry. For
background and the more general Alexandrov-Fenchel inequality, we
refer to Schneider's monograph \cite{Sch}.

There are various reverse inequalities that complement
\eqref{eqn:Alexandrov} or some of its special cases. K. Ball's reverse
isoperimetric inequality shows that any convex body $A$ has an affine
image $\wtd{A}$ such that
\begin{equation}
  \label{eqn:Ball_reverse}
  \frac{V_{n-1}(\wtd{A})}{V_{n-1}(B)}\ls c_n
  \left(\frac{V_n(\wtd{A})}{V_n(B)}\right)^{\frac{n-1}{n}},
\end{equation}
where $c_n$ is a constant which is attained when $A$ is a simplex (and
when $A$ is a cube if one considers only origin-symmetric convex bodies)
\cite{Ba}. A reverse form of Urysohn's inequality can be obtained by a
result of Figiel and Tomczak-Jaegermann \cite{FTJ}: any symmetric
convex body $A$ has a linear image $\wtd{A}$ satisfying
\begin{equation}
  \label{eqn:Pisier_reverse}
  \frac{V_1(\wtd{A})}{V_1(B)}\ls
  CK(A)\left(\frac{V_n(\wtd{A})}{V_n(B)}\right)^{\frac{1}{n}},
\end{equation}
where $C$ is an absolute constant and $K(A)$ denotes the $K$-convexity
constant of $\R^n$ equipped with the norm $\| \cdot \|_A$ associated
to $A$ (see also \cite[Ch. 6]{AGM}).  A fundamental theorem of
Pisier \cite{Pis-remMau} gives $K(A)\ls C\log d(A, B)$, where $d$
denotes Banach-Mazur distance. By John's theorem \cite{Jo}, one always
has $d(A,B)\ls \sqrt{n}$ and thus $K(A)\ls C\log n$.  In a related
direction, by a result of Milman \cite{M-ibm}, any symmetric convex
body $A$ admits a linear image $\wtd{A}$ such that
\begin{equation}
  \label{eqn:Milman_reverse}
  \left(\frac{V_{n/2}(\wtd{A})}{V_{n/2}(B)}\right)^{\frac{2}{n}}\ls
  c_1\left(\frac{V_n(\wtd{A})}{V_n(B)}\right)^{\frac{1}{n}},
\end{equation}
where $c_1$ is an absolute constant.  The latter is based on the
existence of Milman's ellipsoid, which in turn is intimately connected
to the reverse Blasckhe-Santal\'{o} inequality \cite{BM} and the
reverse Brunn-Minkowski inequality \cite{M-ibm} (see also
\cite[Ch. 7]{Pis-book}). Each of \eqref{eqn:Ball_reverse},
\eqref{eqn:Pisier_reverse} and \eqref{eqn:Milman_reverse} share the
common feature that to get reverse inequalities, one needs to consider
affine (or linear) images of the convex body.  Note that
\eqref{eqn:Ball_reverse} is a sharp inequality while
\eqref{eqn:Pisier_reverse} and \eqref{eqn:Milman_reverse} are {\it
  isomorphic} reversals in that they hold up to constants without
establishing extremizers.  Moreover, \eqref{eqn:Pisier_reverse} is a
quantitative statement in the sense that a parameter associated with
$A$ quantifies the tightness of the reversal.

All of the reversed inequalities mentioned so far involve $V_n(A)$.
Concerning the generalized Urysohn inequality, which compares $V_1(A)$
with $V_k(A)$ for $1\ls k\ls n$, Hug and Schneider \cite{HS} have
proved that for any zonoid $A$, there is a linear image $\wtd{A}$ such
that
\begin{equation}
  \label{eqn:HS_reverse}
  \frac{V_1(\wtd{A})}{V_1(B)} \ls
  c(n,k)\left(\frac{V_k(\wtd{A})}{V_k(B)}\right)^{\frac{1}{k}},
\end{equation}
where $c(n,k)$ is a constant that is obtained when $A$ is a
parallelpiped.  As with Ball's inequality, \eqref{eqn:HS_reverse} is
sharp.  The case $k=n$ was proved earlier by Giannopoulos, Milman and
Rudelson \cite{GMR}.

Our first result is a quantitative reversal involving $V_1(A)$ and
$V_k(A)$. We show that when $A$ is symmetric, \eqref{eqn:Alexandrov}
may be reversed up to a new parameter associated with $A$, studied
recently in \cite{PVsd} and \cite{PVvar}. Specifically, let $h_A$
denote the support function of $A$ and let $g$ be a standard Gaussian
random vector in $\R^n$. We define a normalized variance of the random
variable $h_A(g)$ as follows: $$\beta_{*}(A) =\frac{{\rm
    Var}(h_A(g))}{(\mathbb E h_A(g))^2},$$ where $\mathbb E$ denotes
expectation and ${\rm Var}$ is the variance. With this notation, we
have the following theorem.

\begin{theorem}
  \label{thm:main1}
   There exists a constant $c>0$ such that if $A$ is a symmetric
   convex body in $\R^n$ and $1\ls k \ls c/ \beta_\ast (A)$,
   then
   \begin{equation}
     \label{Reverse-AF}
      \frac{V_1(A)}{V_1(B)} \ls \left(1+c \sqrt{k\beta_{\ast} \log
        \left(\frac{e}{k\beta_\ast}\right)} \right)
      \left(\frac{V_k(A)}{V_k(B)}\right)^{1/k}.
   \end{equation}
 \end{theorem}

For comparison purposes, it will be convenient to write $$W_{[k]}(A) =
\left(\frac{V_k(A)}{V_k(B)}\right)^{1/k},$$ which is simply the radius
of a Euclidean ball having the same $k$-th intrinsic volume as
$A$. Then \eqref{eqn:Alexandrov} says that $k \mapsto W_{[k]}(A)$ is
non-increasing, while the Hug-Schneider result
\eqref{eqn:HS_reverse} for zonoids implies that
\begin{equation}
  W_{[1]}(A)\ls \left(1+O\left(\frac{k}{n}\right)\right) W_{[k]}(A).
\end{equation}
For our normalization, quantitative reversals comparing $W_{[n]}(A)$
with $W_{[k]}(A)$ (as opposed to $W_{[k]}(A)$ with $W_{[1]}(A)$) are
somewhat easier tasks to achieve. For example, just using set
inclusions and monotonicity of mixed volumes one has
\begin{equation*}
  W_{[n-k]}(A) \ls d^{\frac{k}{n-k}} W_{[n]}(A) \ls \left(1+ O\left(
  \frac{k \log d}{n-k} \right) \right)W_{[n]}(A),
\end{equation*}
as long as $k\ls \frac{n}{1+\log d}$, where $d=d_G(A)$ is the
geometric distance between $A$ and $B$ (i.e., the ratio of the
circumradius of $A$ over the inradius of $A$). Thus we focus
on comparisons between $W_{[k]}(A)$ and $W_{[1]}(A)$ in this paper.

Theorem \ref{thm:main1} combines several features of the
aforementioned inequalities: one has a quantitative reversal of
\eqref{eqn:Alexandrov} depending on the parameter $\beta_{\ast}(A)$.
Unlike the reverse Urysohn inequality \eqref{eqn:Pisier_reverse},
\eqref{Reverse-AF} holds on an {\it almost isometric scale} as opposed
to an isomporhic one.

To explain some of the ideas behind Theorem \ref{thm:main1}, recall
that $W_{[k]}(A)$ can be expressed through Kubota's integral recursion
(e.g. \cite[Ch. 5]{Sch}) via
\begin{equation}
  \label{eqn:Kubota}
   W_{[k]}(A) = \left(\frac{1}{\omega_k}\int_{G_{n,k}}\abs{P_E A} \,
   d\nu_{n,k}(E)\right)^{\frac{1}{k}},
\end{equation}
where $\omega_k$ is the volume of the Euclidean unit ball in $\R^k$,
$G_{n,k}$ is the Grassmannian of $k$-dimensional subspaces of $\R^n$,
equipped with the Haar probability measure $\nu_{n,k}$, $P_E$ denotes
the orthogonal projection onto $E$ and $|\cdot|$ denotes volume (on
the subspace $E$).  Thus our interest is in tight lower bounds for the
volume of random projections of $A$. By Milman's random version of
Dvoretzky's theorem \cite{Mil-dvo}, one has the following almost
isometric inclusions
\begin{equation}
  \label{eqn:Milman}
  (1-\eps)W_{[1]}(A)P_E B \subseteq P_E A \subseteq
  (1+\eps)W_{[1]}(A)P_E B,
\end{equation}
for a random subspace $E\in G_{n,k}$ provided $k\ls c(\eps)k_{*}(A)$,
where $k_{*}(A)$ denotes the Dvoretzky dimension (the definition is
recalled in \S 3). The inclusions in \eqref{eqn:Milman} explain the
almost constant behavior of $k\mapsto W_{[k]}(A)$ for $k$ up to
$k_*(A)$.  Theorem \ref{thm:main1} goes further in that this
near-constant behavior actually extends for dimensions $k$ up to
$c/\beta_*(A)$. In general, $k_{*}(A)\ls c/\beta_{*}(A)$, while for
some convex bodies, $c/\beta_{*}(A)$ is significantly larger than
$k_{*}(A)$. An earlier indication of this phenomenon is suggested by
work of Klartag and Vershynin in \cite{KV}. They proved that the lower
inclusion in \eqref{eqn:Milman} on an isomorphic scale, i.e.,
\begin{equation}
  \label{eqn:KV}
  c_1 W_{[1]}(A) P_E B \subseteq P_EA,
\end{equation}
can hold for subspaces $E$ of significantly larger dimensions,
governed by a different parameter $d_{*}(A)$ which satisfies
$d_{*}(A)\gr c_2k_{*}(A)$, where $c_1,c_2$ are absolute constants (the
precise definition of $d_*(A)$ is in \S \ref{subsection:norms}). In
particular, they noted the following striking example: for $A=B_1^n$,
the unit ball in $\ell_1^n$, one has $d_*(A)\simeq n^{0.99}$ while
$k_*(A)\simeq \log n$.  The behavior of $\beta_{*}(A)$ has been
studied in \cite{PVsd} and \cite{PVvar} in connection with almost
isometric Euclidean structure and concentration for convex functions.
Theorem \ref{thm:main1} shows that $\beta_{*}(A)$ also plays a
significant role in reversing \eqref{eqn:Alexandrov}.

More generally, we also show that $\beta_{*}(A)$ arises in
multi-dimensional concentration inequalities. In view of Kubota's
formula \eqref{eqn:Kubota}, Theorem \ref{thm:main1} concerns the
expectation of the random variable
$${\rm vrad}(P_EA):=\left( |P_EA| / \omega_k\right)^{1/k},$$ where $E$
is a random subspace distributed according to $\nu_{n,k}$.  For
families of convex bodies $A=A_n\subseteq \R^n$ with $n$ increasing
(and $k$ fixed), it is natural to study distributional properties of
${\rm vrad}(P_EA)$. For example, when $A_n$ is the cube $[-1,1]^n$,
${\rm vrad}(P_EA)$ is studied in \cite{PPZ} and a central limit
theorem is proved.  Here we treat concentration inequalities for
arbitrary symmetric convex bodies. In this way, the next theorem can
be seen as a more quantitative study of the intrinsic volumes.

\begin{theorem}
  \label{thm:main2}
  Let $A$ be a symmetric convex body in $\R^n$ and let $1\ls k \ls
  c/\beta_\ast(A)$. Then for all $\varepsilon > c_1' \sqrt{k\beta_\ast(A)
    \log(\frac{e}{k\beta_\ast(A)})}$,
  \begin{align}
    \label{conc1}
    \nu_{n,k} \left( E\in G_{n,k} : {\rm vrad}(P_EA) \gr (1+\varepsilon)
    W_{[k]}(A) \right)\ls C_1\exp \left(-c_1
    \varepsilon^2kk_\ast(A)\right);
  \end{align}
  moreover, if $c_2' \sqrt{ k\beta_\ast(A)
    \log(\frac{e}{k\beta_\ast(A)})} <\varepsilon <1$,
  \begin{align}
    \label{conc2}
    \nu_{n,k} \left( E\in G_{n,k} : {\rm vrad}(P_EA) \ls
    (1-\varepsilon) W_{[k]}(A) \right)\ls C_2\exp(-c_2 \varepsilon^2 /
    \beta_\ast(A)),
  \end{align}where $c_i,C_i,c_i'>0$, $i=1,2$, are  absolute constants.
\end{theorem}

If we take $k=1$ in Theorem \ref{thm:main2}, then $E={\rm
  span}(\theta)$ for some $\theta$ on the unit sphere $S^{n-1}$ and
${\rm vrad}(P_EA)=h_{A}(\theta)$, while
$W_{[k]}(A)=\int_{S^{n-1}}h_A(\theta)d\sigma(\theta)$.  Thus
\eqref{conc1} recovers the standard concentration estimate on the
sphere in terms of the Lipschitz constant of the support function
$h_A$ of $A$ (up to constants), e.g., \cite[Ch. 2]{MS}.
Similarly, \eqref{conc2} recovers the new concentration inequality in
terms of variance of the support function from \cite{PVsd} (stated
below in Theorem \ref{thm:PV_variance}).  Thus Theorem \ref{thm:main2}
is a multi-dimensional extension of the latter results. Both Theorems
\ref{thm:main1} and \ref{thm:main2} are based on new tight reverse
H\"{o}lder inequalities for the random variables ${\rm vrad}(P_EA)$
and $w(P_EA)$.  These improve the standard estimates following from
the concentration of measure phenomenon in the current literature
(this is discussed in \S \ref{subsection:norms}).

We conclude the introduction with some examples where Theorem
\ref{thm:main1} gives the largest possible range of dimensions for the
almost-constant behavior in \eqref{eqn:Alexandrov}.  Recall that a
Borel measure $\mu$ on $S^{n-1}$ is said to be isotropic if the
covariance matrix of $\mu$ is the identity matrix. For any such
measure we associate the family of the $L_q$-zonoids $\{ Z_q(\mu)
\}_{q\gr 1}$ which are defined through their support function:
\begin{align*}
h_{Z_q(\mu)}(x) = \left( \int_{S^{n-1}} | \langle x,\theta \rangle|^q
\, d\mu(\theta) \right)^{1/q}, \quad x\in \R^n.
\end{align*}

\begin{corollary} \label{cor:app1}
  Let $1\ls q< \infty$. Then there is a constant $c_q>0$ such that if
  $k\ls c_qn$ and $\mu$ is an isotropic Borel measure on $S^{n-1}$, then
  \begin{equation}
    \left(1-\sqrt{\frac{c_qk}{n}\log \frac{n}{c_qk}}\right)W_{[1]}(Z_q(\mu))
    \ls W_{[k]}(Z_q(\mu)) \ls W_{[1]}(Z_q(\mu)).
  \end{equation}
\end{corollary}

Lastly, the restriction to symmetric convex bodies in this paper seems
to be inherent in the tools used in the proofs. However, we do not
believe that symmetry is essential for such reverse inequalities.

The rest of the paper is organized as follows: In Section
\ref{section:notation}, we fix the notation and we provide necessary
background information. In Section \ref{section:Probab}, we recall
some auxiliary results from asymptotic convex geometry and from the
concentration of measure for norms on Euclidean space. Some basic
probabilistic facts are also considered.  Finally, in Section
\ref{section:multi} we present the proofs of our main results.

\section{Notation and background material}

\label{section:notation}

We work in $\R^n$ equipped with the usual inner-product $\langle
\cdot, \cdot \rangle$ and Euclidean norm $\norm{x}_{2}:= \sqrt{
  \langle x,x \rangle}$ for $x\in \R^n$; $B_{2}^{n}$ is the Euclidean
ball of radius $1$; $S^{n-1}$ is the unit sphere, equipped with the
Haar probability measure $\sigma$. For Borel sets $A\subseteq \R^n$,
we use $V_n(A)$ (or $\abs{A}$) for the Lebesgue measure of $A$;
$\omega_{n}$ for the Lebesgue measure of $B_{2}^{n}$.  The
Grassmannian manifold of all $n$-dimensional subspaces of $\R^n$ is
denoted by $G_{n,k}$, equipped with the Haar probability measure
$\nu_{n,k}$.  For a subspace $E\in G_{n,k}$, we write $P_E$ for the
orthogonal projection onto $E$.

Throughout the paper we reserve the symbols $c,c_1,c_2,\ldots$ for
absolute constants (not necessarily the same in each occurrence). We
use the convention $S\simeq T$ to signify that $c_{1} T\ls S\ls
c_{2}T$ for some positive absolute constants $c_1$ and $c_2$. We also
assume that $n$ is larger than a fixed absolute constant.  By
adjusting the constants involved one can always ensure that the
results to hold for all $n$.

A convex body $K\subseteq \R^n$ is a compact, convex set with
non-empty interior.  The support function of a convex body $K$ is
given by
\begin{equation*}
  h_K(y) =\sup\{\langle x, y \rangle: x\in K\} , \; y\in \R^n.
\end{equation*} 
We say that $K$ is (origin) symmetric if $K=-K$. For a symmetric
convex body $K$ the polar body $K^{\circ}$ is defined by
$$ K^{\circ}:= \{ x\in \mathbb R^n: |\langle x, y\rangle | \ls 1, y
\in K\}. $$ For $p\neq 0$, we define the $p$-generalized mean width of
$A$ by
\begin{equation}
  \label{eqn:wp}
  w_p(K):= \left(\int_{S^{n-1}} h_{K}^{p}(\theta)
  d\sigma(\theta)\right)^{1/p}.
\end{equation}
The circumradius of $K$ is defined by $R(K)=\max_{\theta\in S^{n-1}}
h_K(\theta)=\max_{x\in K}\|x\|_2$.  Note that $R(K)=w_\infty(K):=
\lim_{p\to \infty} w_p(K)$. In addition, we denote by $r(K)$ the
inradius of $K$, i.e. $r(K)= \min_{\theta \in S^{n-1}}
h_K(\theta)$. Again, we have: $r(K)=w_{-\infty}(K):= \lim_{p\to
  \infty} w_{-p}(K)$. Note that $r(K^\circ)=1/R(K)$. Similarly, if $\|
\cdot \|_{K}$ is the norm induced by $K$ we define, for $p\neq 0$,
$$ M_{p}(K):= \left( \int_{S^{n-1}}
\|\theta\|_{K}^{p}d\sigma(\theta)\right)^{\frac{1}{p}} . $$ Note that
$M_{p}(K^{\circ})= w_{p}(K)$, by definition. We simply write
$w(K):=w_1(K)$ and $M(K):= M_{1}(K)$.

The intrinsic volumes of a convex body $K\subseteq \R^n$ can be defined
via the Steiner formula for the outer parallel volume of $K$:
$$ |K+ t B_2^n| = \sum_{k=0}^n \omega_k V_{n-k}(K)t^k, \quad t>0.
$$ Here $V_k$, $k=1,\dots,n$, is the $n$-th intrinsic volume of $K$
(we set $V_0 \equiv 1$). $V_n$ is volume, $2V_{n-1}$ is surface area
and $\frac{\omega_{n-1}}{n\omega_n}V_1=w=w_1$ is the mean width (as we
have defined in \eqref{eqn:wp}). Intrinsic volumes are also referred
to as quermassintegrals (under an alternate labeling and
normalization). For further background, see \cite[Ch. 4]{Sch}. Here we prefer
to work with a different normalization, similar to that used in
\cite{DP}, \cite{PaPiv}. As in the introduction, for a convex body
$K\subseteq \R^n$ and $1\ls k \ls n-1$, we write
$$ W_{[k]}(K) := \left(\frac{1}{\omega_k} \int_{G_{n,k}} |P_E K| \,d
\nu_{n,k}(E)\right)^{1/k}.$$ 
We will need the following generalization of this definition:
for $p\neq 0$ we write
$$ W_{[k,p]}(K) := \left( \frac{1}{\omega_k^p} \int_{G_{n,k}} |P_{E}
K|^{p} \, d\nu_{n,k}(E)\right)^{\frac{1}{pk}}.$$ Note that by Kubota's
integral formula,
\begin{equation}
\label{V-W}
V_{k}(K)= {n\choose k}\frac{\omega_{n}}{ \omega_{n-k}}W_{[k]}^{k}(K).
\end{equation}
We also set $W_{[n]}(K)={\rm vrad}(K) := \left(
\frac{V_n(K)}{V_n(B_2^n)} \right)^{1/n}$. For ease of reference, we
will also explicitly recall Urysohn's inequality which is the endpoint
inequality from \eqref{eqn:Alexandrov}:
\begin{equation}
  \label{eq:Ury}
  w(K)= w_1(K)= W_{[1]}(K) \gr W_{[n]}(K) = {\rm vrad}(K)= \left(
  \frac{|K|}{|B_{2}^{n}|}\right)^{1/n}.
\end{equation}

\section{Probabilistic and geometric tools}
\label{section:Probab}

We start with a few elementary lemmas about moments of random
variables. Since we need some refinements of standard inequalities, we
include somewhat detailed proofs.  We then combine these with Gaussian
concentration inequalities to prove new sharp reverse-H\"{o}lder
inequalities for norms of random vectors.

\subsection{Centered and noncentered moments of random variables}

\label{subsection:moments}

We begin with the following standard fact.

\begin{proposition}
  \label{prop:var-med}
  Let $\xi$ be a random variable on a probability space $(\Omega, \cal
  A, \mathbb P)$ with $\xi \in L_2(\Omega)$.  If $m= {\rm med}(\xi)$
  is a median of $\xi$, then
  \begin{align*}
    \mathbb E |\xi-m| \ls \sqrt{ {\rm Var}(\xi) }.
  \end{align*} 
\end{proposition}

\begin{proof}
  Recall that $\inf_{\lambda \in \mathbb R} \mathbb E|\xi -\lambda|=
  \mathbb E |\xi -m|$. Thus, by the Cauchy-Schwarz
  inequality,
 \begin{align*}
   \sqrt{ {\rm Var}(\xi) } \gr
   \mathbb E |\xi -\mathbb E \xi| \gr \mathbb E |\xi -m|.
 \end{align*} 
 \end{proof}
 
\begin{lemma}
 Let $\xi$ be a random variable on a probability space $(\Omega, \cal
 A, \mathbb P)$ with $\xi \in L_p(\Omega)$, $p\gr 2$ and let $k\in
 \mathbb N$ with $2\ls k\ls p$. Then, for any $a\neq 0$,
\begin{align*}
  \frac{\mathbb E \xi^k}{a^k} =1 +\sum_{s=1}^k {k \choose s}
  \frac{\mathbb E (\xi-a)^s}{a^s}.
\end{align*}
In particular, if $k\gr 2$ and we take $a=\mu:=\mathbb E\xi \neq 0$,
then
\begin{align*}
  \frac{\mathbb E \xi^k}{\mu^k} =1 +\sum_{s=2}^k {k \choose s}
  \frac{\mathbb E (\xi-\mu)^s}{\mu^s}.
\end{align*}
\end{lemma}

\begin{proof}
  Using the binomial expansion, we have
  \begin{align*}
    \mathbb E\xi^r = \mathbb E [(\xi-a) +a]^r = \sum_{s=0}^r {r
      \choose s} \mathbb E(\xi-a)^s a^{r-s} = a^r \left[1+\sum_{s=1}^r
      {r \choose s} \frac{\mathbb E(\xi-a)^s}{a^s}\right],
  \end{align*}
  for all positive integers $r\gr 1$. The result follows.
\end{proof}

\begin{proposition} \label{prop:conc-moms}
Let $\xi$ be a non-negative random variable with $\mathbb E \xi=\mu>0$
and let $A\gr 1$, $k\gr 1$ and $a >0$ be constants with
$P(|\xi-\mu|>t\mu) \ls A e^{-at^2k} $, for all $t>0$. Then for any
$s\gr 2$,
\begin{align}
  \label{centered_moment}
  \frac{ \mathbb E|\xi-\mu|^s}{\mu^s} \ls \left(\frac{As}{ak} \right)^{s/2}.
\end{align}
Moreover, for all $r\gr 1$,
\begin{align}
  \label{noncentered_moment}
  \|\xi \|_r= (\mathbb E \xi^r)^{1/r} \ls \sqrt{ 1 +\frac{CAr}{ak} } \mu,
\end{align} 
 where $C>0$ is an absolute constant.
\end{proposition}

\begin{proof} Observe that
  \begin{align*}
    \mathbb E |\xi-\mu|^s & = s\mu^s \int_0^\infty z^{s-1}
    P(|\xi-\mu|>z\mu) \, dz \ls A s \mu^s \int_0^\infty z^{s-1}
    e^{-az^2k}\, dz \\ & = \frac{As}{2} (ak)^{-s/2} \mu^s
    \int_0^\infty z^{\frac{s}{2}-1} e^{-z}\, dz =
    \frac{A\mu^s}{(ak)^{s/2}} \Gamma\left(\frac{s}{2}+1\right) \\ &\ls
    \frac{A\mu^s}{(ak)^{s/2}} s^{s/2} \ls \mu^s \left(
    \frac{As}{ak}\right)^{s/2},
  \end{align*} where we have used the rough estimate $\Gamma(x+1) < (2x)^x$
  for $x\gr 1$. This completes the proof of the first assertion.

Set $c:=A/a$. Next, using formula \eqref{centered_moment} we have
\begin{align}
\frac{\mathbb E \xi^r}{\mu^r} &\ls 1+ \sum_{s=2}^r {r\choose s}
\frac{\mathbb E |\xi-\mu|^s}{\mu^s} \ls 1+\sum_{s=2}^r {r\choose s}
\left( \frac{cs}{k}\right)^{s/2} \\ &\ls 1+\sum_{s=2}^r \left( \frac{e
  c^{1/2} r}{k^{1/2}}\right)^s \frac{1}{s^{s/2}} \ls
1+\sum_{s=2}^\infty \frac{\theta^s}{s^{s/2}}, \nonumber
\end{align}
where we have used the estimate ${n\choose k}\ls (en/k)^k$ and
$\theta:=ec^{1/2}r/k^{1/2}$. If $r>k$, then
\eqref{noncentered_moment}, follows from \eqref{centered_moment} and
the triangle inequality (and possibly adjusting the constant).  Thus
we consider only $r\ls k$ and distinguish two cases:

\noindent Case i: $\theta<1/2$. In this case, we have
\begin{align*}
\frac{\mathbb E \xi^r}{\mu^r} &\ls 1+ \sum_{s=2}^\infty \theta^s\ls 1+ 2\theta^2.
\end{align*} 

\noindent Case ii: $\theta\gr 1/2$. We write
\begin{align*}
\frac{\mathbb E \xi^r}{\mu^r} &\ls 1+\sum_{s=1}^\infty
\frac{\theta^{2s}}{(2s)^s} +\sum_{s=1}^\infty
\frac{\theta^{2s+1}}{(2s+1)^{s+1/2}} \ls 1+\sum_{s=1}^\infty
\frac{(\theta^2/2)^s}{s!} +\theta \sum_{s=1}^\infty
\frac{(\theta^2/2)^s}{s!} \\ &\ls e^{\theta^2/2} +\theta
e^{\theta^2/2} \ls \exp\left( \theta+\theta^2/2\right) \ls
\exp\left(\frac{5\theta^2}{2}\right).
\end{align*} 
In either case, we have $(\mathbb E \xi^r)^{1/r}\ls \mu
e^{3\theta^2/r}$ and since $r\ls k$, the result follows.
\end{proof}

\begin{remark} {\rm \bf 1.} If $\xi$ satisfies ${\rm Var}(\xi)\gr c_1 \mu^2 / k$
  (the maximal possible lower bound in light of
  \eqref{centered_moment}), a similar argument shows the reverse
  inequality
\begin{align*}
(\mathbb E \xi^r )^{1/r} \gr \left( 1+\frac{c_2 r}{k}\right)\mu,
\end{align*}
for all $2\ls r \ls c_3\sqrt{k}$, where $c_1, c_2, c_3>0$ are
constants depending only on $A,a>0$. Thus \eqref{noncentered_moment}
is essentially tight for $2\ls r \ls c_3\sqrt{k}$.

\noindent {\rm \bf 2.} If $\xi$ has sub-exponential tails, i.e. $P(
|\xi-\mu|>t\mu)\ls A \exp(-atk)$ for all $t>0$, then for all $s\gr 1$,
\begin{align} \label{eq:moms}
 \left(\mathbb E|\xi-\mu|^s\right)^{1/s} \ls \mu \frac{c_1s}{k}.
\end{align}  Moreover, for $1\ls r\ls ck$, we have
\begin{align*} 
\left(\mathbb E \xi^r \right)^{1/r} \ls \left(1+
\frac{c_3r}{k^2}\right) \mu.
\end{align*}
As above, if ${\rm Var} (\xi) \gr c_1' \mu^2/k^2$, then the reverse
estimate also holds, i.e.,
\begin{align*}
\left(\mathbb E \xi^r\right)^{1/r} \gr 
\left(1+\frac{c_2'r}{k^2}\right)\mu,
\end{align*}
for $2\ls r\ls c'k$,  where $c_1', c_2', c,c'>0$ are constants depending only on $A,a>0$.  
\end{remark}

\subsection{Sharp reverse-H\"{o}lder inequalities for norms}
\label{subsection:norms}

We now turn to norms on $\R^n$. If $A$ is a symmetric convex body in
$\R^n$ with norm $\norm{\cdot}_A$, we write $M(A):= \int_{S^{n-1}}
\|\theta\|_{A} d\sigma(\theta) $ and $b(A) := \sup_{\theta\in S^{n-1}}
\|\theta\|_{A}$. Set $v(A):= {\rm Var}_{\gamma_n}\|x\|_A$ and write
$m(A)$ for the median of the function $\| \cdot \|_A$ with respect to
the Gaussian measure $\gamma_n$, i.e.,
\begin{equation} \label{def-median}
 \gamma_n\left( \{x: \| x\|_A \ls m(A)\} \right) \gr \frac{1}{2}
 \ {\rm and} \ \gamma_n\left( \{x: \| x\|_A \gr m(A)\} \right) \gr
 \frac{1}{2} .
 \end{equation} 
For $-n<p \neq 0$, we also define
$$I_{p} (\gamma_n , A) := \left (\int_{\R^n} \| x\|_{A}^{p} \,
d\gamma_n(x) \right)^{\frac{1}{p}}.$$ Using polar coordinates,
\begin{equation} \label{def-Ip}
 I_{p}(\gamma_n , A) = a_{n,p} M_{p}(A), \quad a_{n,p}:= I_p(\gamma_n, B_2^n).
\end{equation}
We will use the standard concentration inequality for $\norm{\cdot}_A$
on $\R^n$ equipped with $\gamma_n$, as well as a recent refinement.

\begin{theorem}
  Let $\norm{\cdot}_A$ be a norm associated with a symmetric convex
  body $A$ in $\R^n$. Then for any $t\gr 0$,
  \begin{equation}
    \label{conc-1}
    \max\big\{ \gamma_n \left(\left\{ x: \norm{x}_A < m(A)-t \right\} \right),
    \gamma_n \left(\left\{ x: \norm{x}_A \gr m(A) +t \right\} \right)\big\} \ls
    \frac{1}{2}e^{ -ct^2/b(A)^2},
 \end{equation}where $c>0$ is an absolute constant.
\end{theorem}

For further background on the latter theorem, see e.g. \cite{Pis},
\cite{MS}. Recently, it has been observed that in the lower small
deviation regime $\{x: \norm{x}_A< m-t\}$, the following refinement
holds \cite{PVsd} (recall that $v(A)\ls b(A)^2$).

\begin{theorem}
  \label{thm:PV_variance}
  Let $\norm{\cdot}_A$ be a norm associated with a symmetric convex
  body $A$ in $\R^n$.  Then for any $t\gr 0$,
  \begin{equation}
    \label{conc-2}
    \gamma_n \left( \{ x: \norm{x}_A \ls m(A)-t \} \right) \ls
    \frac{1}{2}e^{-ct^2/ v(A)},
 \end{equation}where $c>0$ is an absolute constant.
\end{theorem}

For a convex body $A\subseteq \R^n$, the Dvoretzky number $k(A)$ is
the maximum $k\ls n$ such that a $\nu_{n,k}$-random subspace $E$ has
the property that $A\cap E$ is $4$-isomorphic to the Euclidean ball of
radius $\frac{1}{M(A)}$ with probability at least $1/2$.  Milman's
formula (see \cite{Mil-dvo}, \cite{MS}) states that $k(A) \simeq n
\frac{ M(A)^{2}}{ b(A)^{2}}$. Moreover if $A$ is in John's position
then $k(A) \gr c\log{n}$ (see \cite{MS}). We write $k_{\ast}(A) =
k(A^{\circ}) $. In this case Milman's formula becomes
\begin{equation}
  \label{Milman}
  k_{\ast} (A) \simeq n\frac{w(A)^{2}}{ R(A)^{2}}.
\end{equation}

For completeness, we also recall a definition of Klartag and Vershynin
from \cite{KV}. For a symmetric convex body $A\subseteq \R^n$, let
\begin{equation*}
  d(A)= \min(-\log \sigma\{\theta \in S^{n-1}:2\norm{\theta}_A\leq
  M(A)\},n).
\end{equation*}
One can check that $d(A)\gr c k(A)$ (see, e.g., \cite{KV}).  We also
set $d_*(A) = d(A^{\circ})$.

We also define $\beta(A)$ as the {\it normalized variance}, i.e.
\begin{align*}
  \beta(A) = \frac{{\rm Var}_{\gamma_n}\|g\|_A}{ ( \mathbb
    E_{\gamma_n}\|g\|_A)^2},
\end{align*}
where $g$ is an $n$-dimensional standard Gaussian random vector (see
\cite{PVsd} and \cite{PVvar} for related background). We write
$\beta_\ast(A)= \beta (A^\circ)$ and note that $\beta(A) \ls c/k(A)$
(see e.g., \cite{PVvar}). As an application of inequality
\eqref{conc-1} and Proposition \ref{prop:conc-moms} we get
\begin{equation}
  \label{LMS}
  M_{q}(A)\ls M(A) \sqrt{ 1 + \frac{c_1q}{k (A)} },
\end{equation}
for every $q\gr 1$ (see also \cite{PVZ} for an alternative proof which
uses the log-Sobolev inequality). For comparison, we note that similar
reverse H\"{o}lder inequalities have often been stated in the form
\begin{equation}
  \label{LMS_original}
  M_{q}(A)\ls M(A)\left( 1 + \sqrt{\frac{c_1q}{k (A)} }\right).
\end{equation}
See for example, \cite[Statement 3.1]{LMS} or \cite[Proposition 1.10,
  (1.19)]{Led}.  Thus in the range $1\ls q \ls k(A)$, \eqref{LMS}
improves upon \eqref{LMS_original}.

In \cite{PVsd}, using \eqref{conc-2} and a small ball probability
estimate in terms of $\beta(A)$, the following reverse H\"older
inequalities for the Gaussian moments of $x\mapsto \|x\|_A$ are
obtained:
\begin{align}
  \label{eq:rev-neg-moms}
  I_{-q}(\gamma_n,A) \gr m(A) \exp \left(- c_1 \max\{\sqrt{\beta}, q\beta\} \right),
\end{align}
for all $0<q<c_2/\beta$, where $\beta\equiv\beta(A)$.

\medskip

We will also use the following application of Proposition \ref{prop:var-med}.

\begin{lemma}\label{Lemma-m-vs-I}
 Let $A$ be a symmetric convex body in $\R^n$. Then 
\begin{equation}
\label{m vs I}
1 \ls \frac{I_1(\gamma_n,A)}{m(A)} \ls 1 + c\sqrt{ \beta(A) } .
\end{equation}
\end{lemma}

\begin{proof}
  The left-hand side follows from the fact that $x\mapsto \|x\|_A$ is
  convex combined with the main result of \cite{Kwa}. The right-hand
  side follows by Proposition \ref{prop:var-med} and the definition of
  $\beta$ (and the standard fact that $m(A)\simeq \mathbb{E}_{\gamma_n}
  \|g\|_A$).
\end{proof}

\begin{proposition}\label{Prop3.1}
 Let $A$ be a symmetric convex body in $\R^n$. Then for all $q\gr 1$,
 \begin{equation}
    \label{3.1}
    w_{q}(A)\ls w(A) \sqrt{ 1 + \frac{cq}{ k_{\ast}(A) } } ,
    \end{equation}
    where $c>0$ is an absolute constant.  Moreover, for all
    $0<q<c_2/\beta_\ast(A)$,
 \begin{equation}\label{3.2}
w_{-q}(A) \gr \left( 1- c_1\min\left\{ \frac{q}{k_\ast(A)} , \max \left\{\sqrt{\beta_\ast(A)}, q\beta_\ast(A) \right\} \right\}\right) w(A),
\end{equation} 
where $c_1, c_2>0$ are absolute constants. 
\end{proposition}

\begin{proof}
  Inequality \eqref{3.1} is simply a reformulation of \eqref{LMS}. For
  proving \eqref{3.2} first we combine \eqref{eq:rev-neg-moms} with
  Proposition \ref{prop:var-med} to get
  \begin{align*}
    \frac{I_{-q}(\gamma_n, A^\circ)}{I_1(\gamma_n,A^\circ)} \gr 1-c_1
    \max \left\{ \sqrt{ \beta_\ast(A)} , q\beta_\ast(A) \right\} ,
  \end{align*} for $0<q< c_2/\beta_\ast(A)$. Furthermore, it is known that
  \begin{align*}
    \frac{I_{-q}(\gamma_n,A^\circ)}{I_1(\gamma_n,A^\circ)} \gr 1- \frac{c_3q}{k_\ast(A)},
  \end{align*} for all $0< q <c_4k_\ast(C)$. (A proof of this fact can be 
  found e.g. in \cite{PVZ}). Using \eqref{def-Ip} we find 
  \begin{align*}
    \frac{I_{-q} (\gamma_n,A^\circ)}{I_1(\gamma_n, A^\circ)} =  \frac{I_{-q} (\gamma_n,B_2^n) w_{-q}(A)}{I_1(\gamma_n, B_2^n) w(A)} \ls
    \frac{w_{-q}(A)}{w(A)}. 
  \end{align*} Combining all these estimates we arrive at \eqref{3.2}.
\end{proof}

\begin{theorem}[Concentration for mean width] \label{thm:conc-mw}
Let $A$ be a symmetric convex body in $\R^n$ and let $1\ls k \ls
n-1$. Then for all $t>0$,
\begin{align} \label{eq:conc-mw}
\nu_{n,k} \left( \left\{ E\in G_{n,k} : |w(P_EA)-w(A)|>tw(A)
\right\}\right) \ls c_1\exp (-c_2t^2 kk_\ast(A)).
\end{align} Moreover, for all $r>0$, 
\begin{align*}
\left(\int_{G_{n,k} } w(P_EA)^r \, d\nu_{n,k}(E) \right)^{1/r} \ls
w(A) \sqrt{1+\frac{c_1 r}{kk_\ast(A)}}.
\end{align*} 
\end{theorem}

\begin{proof}
  For a proof of the first part, we refer the reader to
  \cite[Prop. 3.9]{PVsd} (which is stated in the Gaussian setting);
  see also \cite[\S 6]{PVvar}. The second part follows from the
  concentration estimate \eqref{eq:conc-mw} and Proposition
  \ref{prop:conc-moms}.
\end{proof}

The next Lemma has its origins in \cite{Klartag_PAMS},
\cite{KV}. However, our formulation takes into account the order of
magnitude of the constants involved; see \cite[pg. 14]{PVsd} for a
proof in the Gaussian setting and \cite{PVvar} for an alternative
proof.

\begin{lemma} [Dimension lift] \label{lem:dim-lift}
Let $A$ be a symmetric convex body in $\mathbb R^n$ and let $1\ls k
\ls n-1$. Then for any $q\gr k$, we have
\begin{align*}
\left( \int_{G_{n,k}} [r(P_EA)]^{-q} \, d\nu_{n,k}(E) \right)^{1/q}
\ls \left( 1+ \frac{ck}{q} \log \left( \frac{eq}{k} \right) \right)
\frac{w(A)}{[w_{-3q}(A)]^2},
\end{align*} where $c>0$ is an absolute constant.
\end{lemma}

\section{Multidimensional concentration for the volume of projections}
\label{section:multi}

Now we turn to proving the main results of the paper. First we study
the almost constant behavior of the mapping $E\mapsto |P_EA|, \; E\in
G_{n,k}$ by establishing reverse-H\"older inequalities for positive
and negative moments. Second, we prove the deviation inequalities
announced in Theorem \ref{thm:main2}.

\subsection{Reverse-H\"older inequalities for generalized intrinsic volumes}
\label{section:reverse_iv}

We start with an inequality for positive moments, which follows from
Uryshon's inequality \eqref{eq:Ury} and Theorem \ref{thm:conc-mw}.

\begin{proposition}\label{Prop-Uryshon}
 Let $A$ be a symmetric convex body in $\R^n$, $1\ls k \ls n-1$. Then,
 for all $p >0$,
\begin{equation}
\label{Uryshon-q}
W_{[k,p]}(A) \ls  w(A) \sqrt{1+\frac{c_1p}{k_\ast(A)}}.
\end{equation} 
\end{proposition}

\begin{proof}
Using Uryshon's inequality \eqref{eq:Ury} we have
$$ W_{[k,p]}(A) = \left( \int_{G_{n,k}} {\rm vrad}(P_E A)^{pk}
d\nu_{n,k}(E) \right)^{\frac{1}{pk}} \ls \left(\int_{G_{n,k} }
w(P_EA)^{pk} d\nu_{n,k}(E) \right)^{\frac{1}{pk}}. $$ Now we apply
Theorem \ref{thm:conc-mw} to get
\begin{align*}
\left(\int_{G_{n,k} } w(P_EA)^{pk} \, d\nu_{n,k}(E) \right)^{\frac{1}{pk}}
\ls w(A) \sqrt{1+\frac{c_1 pk}{kk_\ast(A)}}.
\end{align*} 
\end{proof}

\begin{proposition} \label{Prop-Quer}
  Let $A$ be a symmetric convex body in $\R^n$. Let $2 \ls k \ls\frac{
    c_1 }{\beta_\ast(A) }$. Then,
  \begin{align}
    \label{Quer}
    W_{[k,- p]}(A)\gr \left(1- c_2 \max\left\{ \sqrt{ k\beta_\ast \log \left(\frac{e}{k\beta_\ast}\right)},pk\beta_\ast \right\}\right)w(A)
  \end{align} for all $0< p\ls \frac{c_3}{k\beta_\ast}$.
\end{proposition}

\begin{proof}
  We may assume that $p \gr 1$ and $pk \ls c_1'/\beta_\ast$. Then for
  any $pk \ls q \ls c_1'/\beta_\ast$ (which will be suitably chosen
  later) we have
\begin{align} \label{eq:quer-inrad}
W_{[k,-p]}(A) \gr W_{[k,-q/k]} (A) &\gr \left( \int_{G_{n,k}}
[r(P_EA)]^{-q} \, d\nu_{n,k}(E) \right)^{-1/q}.
\end{align}
Using Lemma \ref{lem:dim-lift} and \eqref{3.2}, \eqref{eq:quer-inrad} becomes:
\begin{align*}
W_{[k,-p]}(A) \gr w(A) \exp \left(- \frac{ck}{q} \log
\left(\frac{eq}{k}\right) - \tau(q)\right),
\end{align*} for all $pk \ls q\ls c_1'/\beta_\ast$, where $\tau(q)= \min \{ q/k_\ast , \max \{\sqrt{\beta_\ast}, q\beta_\ast\} \}$.
The choice $q = \sqrt{ \frac{k}{\beta_\ast}\log ( \frac{e
  }{k\beta_\ast} ) }$ yields the estimate:
\begin{align*}
\frac{ck}{q} \log \left(\frac{eq}{k}\right) + \tau(q) \simeq \sqrt{
  k\beta_\ast \log \left(\frac{e}{k\beta_\ast} \right)},
\end{align*} for all $0<p \ls \frac{q}{k} = \sqrt{ \frac{1}{k \beta_\ast}\log ( \frac{e }{k \beta_\ast} ) }$. On the other hand
when $p \gr \sqrt{ \frac{1}{k \beta_\ast}\log ( \frac{e }{k
    \beta_\ast} ) }$, we choose $q=pk$ to obtain the estimate:
\begin{align*}
\frac{ck}{q} \log \left(\frac{eq}{k}\right) + \tau(q) \simeq \max
\left\{ \frac{\log p}{p}, pk\beta_\ast \right\} \simeq pk\beta_\ast.
\end{align*} Combining the above we get the result.
\end{proof}

Theorem \ref{thm:main1} is an immediate consequence of the following
corollary (with possibly adjusting the constant $c$).

\begin{corollary} \label{Cor-Reverse-AF}
There exists $c>0$ such that if $A$ is a symmetric convex body in
$\R^n$ and $1\ls k \ls c/ \beta_\ast (A)$, then
\begin{equation}
 \left(1-c \sqrt{ k\beta_\ast \log \left(\frac{e}{k\beta_\ast} \right)
 }\right) w(A) \ls W_{[k]}(A) \ls w(A).
\end{equation}
\end{corollary}

\begin{proof}
 The right-hand side inequality follows from Urysohn's inequality
 \eqref{eq:Ury} applied for the body $P_EA$. The left-hand side
 inequality follows from the fact that $W_{[k]}(A) \gr W_{[k,-1]}(A)$
 and \eqref{Quer}.
\end{proof}

\begin{theorem} \label{thm:rev-H}
 There exist $c, c_{1}, c_{2}, c_{3}, c_{4}$ such that the following
 holds: Let $A$ be a symmetric convex body in $\R^n$. Let $2\ls
 k \ls c /\beta_{\ast}(A)$ and $0<p< c_{1} k_{\ast}(A)$. Then,
\begin{equation} \label{cor-1}  
 \frac{W_{[k,p]}(A)}{W_{[k]}(A)} \ls 1+ c_{2}
 \max\left\{\frac{p}{k_\ast} , \sqrt{ k\beta_\ast \log
   \left(\frac{e}{k\beta_\ast} \right) } \right\}.
\end{equation}
Moreover, if $2\ls k \ls c_3 / \beta_\ast(A)$ and $0< p <
\frac{c}{k\beta_\ast(A)}$, we have
\begin{equation} \label{cor-2}
\frac{ W_{[k,-p]}(A)}{W_{[k]}(A)} \gr 1-c_4\max\left\{ pk\beta_\ast,
\sqrt{ k\beta_\ast \log \left( \frac{e}{k\beta_\ast} \right)}
\right\}.
\end{equation}
\end{theorem}

\begin{proof} 
 Let $0<p< ck_{\ast}(A)$. Then using (\ref{Uryshon-q}) and
 (\ref{Reverse-AF}), we get
\begin{align*} 
\frac{W_{[k,p]}(A)}{W_{[k]}(A)} \ls \frac{ \left( 1+\frac{cp}{k_\ast
    (A)} \right) w(A)}{ \left( 1- c' \sqrt{ k\beta_\ast \log(
    \frac{e}{k\beta_\ast} )} \right) w(A)} \ls 1+c
\max\left\{\frac{p}{k_\ast} , \sqrt{ k\beta_\ast
  \log(\frac{e}{k\beta_\ast})} \right\}.
\end{align*}
Moreover, using (\ref{Quer}) and (\ref{Reverse-AF}) we obtain
\begin{eqnarray*}
\frac{ W_{[k,-p]}(A)}{W_{[k]}(A)} &\gr& \frac{ \left(1-c\max
  \left\{pk\beta_\ast, \sqrt{ k\beta_\ast \log \left(
    \frac{e}{k\beta_\ast} \right)}\right\} \right) w(A)}{w(A)}.
\end{eqnarray*} 
\end{proof}

\subsection{Deviation inequalities}

In this section we prove Theorem \ref{thm:main2}. We consider the
upper and lower inequalities separately.

\begin{theorem}
  Let $A$ be a symmetric convex body in $\R^n$ and let $1\ls k \ls
  c/\beta_\ast(A)$. Then for all $\varepsilon \gr c\sqrt{k\beta_\ast
    \log(\frac{e}{k\beta_\ast})}$,
  \begin{align*}
    \nu_{n,k} \left( E\in G_{n,k} : {\rm vrad}(P_EA) \gr
    (1+\varepsilon) W_{[k]}(A) \right)\ls C_1\exp \left(-c_1
    \varepsilon^2kk_\ast(A)\right).
  \end{align*} 
\end{theorem}

\begin{proof}
 For $\varepsilon \gr c \sqrt{k\beta_\ast
   \log(\frac{e}{k\beta_\ast})}$, we apply Corollary
 \ref{Cor-Reverse-AF} to get
  \begin{align*}
    \left \{E \in G_{n,k} : {\rm vrad}(P_EA)\gr
    (1+\varepsilon)W_{[k]}(A) \right \} \subseteq \left \{E: w(P_EA)
    \gr \left( 1+\frac{\varepsilon}{2} \right) w(A) \right\}.
  \end{align*}
 The result follows if we use the estimate from Theorem
 \ref{thm:conc-mw}.
\end{proof}

\begin{theorem}
Let $A$ be a symmetric convex body in $\R^n$ and let $1\ls k \ls
c/\beta_\ast(A)$. Then for all $c_1\sqrt{ k \beta_\ast
  \log(\frac{e}{k\beta_\ast})} <\varepsilon <1$, 
\begin{align*}
\nu_{n,k} \left( E\in G_{n,k} : {\rm vrad}(P_EA) \ls (1-\varepsilon)
W_{[k]}(A) \right)\ls \exp(-c\varepsilon^2 / \beta_\ast).
\end{align*}

\end{theorem}

\begin{proof}
Let $\varepsilon \in (0,1)$. For any $p\gr
\sqrt{ \frac{1}{k\beta_\ast} \log(\frac{e}{k\beta_\ast})}$ we apply
Markov's inequality and Theorem \ref{thm:rev-H} to get
\begin{eqnarray*}
\nu_{n,k}\left( \left\{ E: {\rm vrad}(P_EA) \ls (1-\varepsilon) W_{[k]}(A) \right\} \right) 
&\ls & e^{-\varepsilon pk} \left(\frac{W_{[k]}(A)}{W_{[k,-p]}(A)} \right)^{pk}\\
&\ls & \exp \left(-\varepsilon pk +cp^2k^2 \beta_\ast \right).
\end{eqnarray*} Choosing $p\simeq \frac{\varepsilon}{k\beta_\ast}$ we get 
the assertion provided that $\varepsilon\gr c_1\sqrt{ k\beta_\ast \log(\frac{e}{k\beta_\ast})}$. 
\end{proof}

\subsection{Application to $L_q$-zonoids}

Here we explain how one obtains Corollary \ref{cor:app1}. Let us first
recall that a Borel measure $\mu$ on $S^{n-1}$ is isotropic if the
covariance matrix of $\mu$ is the identity, or equivalently, for each
$x\in \R^n$,
\begin{align*}
  \|x\|_2^2 = \int_{S^{n-1}} \langle x, \theta \rangle^2 \, d\mu(\theta).
\end{align*}
For such measures $\mu$, we recall upper bounds for
$\beta_{*}(Z_q(\mu))$ proved in \cite{PV-dvoLp}.

\begin{lemma}
  \label{lem:bd-var1}
  There exists an absolute constant $c>0$ such that for any $n$, for
  any $1\ls q<\infty$ and any isotropic Borel measure $\mu$ on
  $S^{n-1}$,
  \begin{align*}
    \beta_\ast(Z_q(\mu)) \ls \frac{e^{cq}}{n}.
  \end{align*}
\end{lemma}

In view of Lemma \ref{lem:bd-var1} and Theorem \ref{thm:main1} we
readily get Corollary \ref{cor:app1}. Note that for $q=1$ this almost
recovers the asymptotic version of Hug-Schneider's reverse inequality
for zonoids.

Also of interest is the case when $A=B_p^n$ and $p=p(n)$. Note that
this is nothing more than $B_p^n =Z_q(\nu)$ with $\nu=\sum_{j=1}^n
\delta_{e_j}$, where $(e_j)_{j\ls n}$ is the standard basis of
$\mathbb R^n$ and $1/p+1/q=1$. If $q=c_0\log n$ for some suitably
chosen absolute constant $c_0\in (0,1)$, we have $k_{*}(B_p^n)\simeq
\log n$ and $\beta_{*}(B_p^n) \simeq n^{-\alpha}$, while the behavior
of $d_{*}(B_p^n)$ is unclear. Precise asymptotic estimates for
$\beta(B_q^n)$ were proved in \cite{PVZ}, which we now recall as they
show the latter lemma is sharp.

\begin{lemma} \label{lem:bd-var2}
  There exist absolute constants $0<c_0<1<C$ such that for any $n\gr
  2$ and for all $1\ls q\ls c_0\log n$,
  \begin{align*}
    \frac{2^q}{Cq^2 n} \ls \beta(B_q^n) \ls \frac{C2^q}{q^2 n}.
  \end{align*}
\end{lemma}

By invoking Lemma \ref{lem:bd-var2} we get a special case of Corollary
\ref{cor:app1} for $B_p^n$'s with $\frac{c_0\log n}{c_0\log n-1} \ls
p\ls \infty$ and for all $n$ large enough.

\bibliography{PPVbib} \bibliographystyle{alpha}

\vspace{.25cm} 

\noindent \begin{minipage}[l]{\linewidth}

  Grigoris Paouris: {\tt grigoris@math.tamu.edu}\\
  Department of Mathematics, Mailstop 3368\\
  Texas A\&M University\\
  College Station, TX 77843-3368\\
  
  Peter Pivovarov: {\tt pivovarovp@missouri.edu}\\
  Mathematics Department\\
  University of Missouri\\ 
  Columbia, MO 65211\\
    
  Petros Valettas: {\tt valettasp@missouri.edu}\\
  Mathematics Department\\
  University of Missouri\\ 
  Columbia, MO 65211\\

\end{minipage}

\end{document}